\documentclass[11pt,a4paper]{amsart}
\usepackage{graphicx,multirow,array,amsmath,amssymb,kotex}

\usepackage{xcolor}
\usepackage{longtable}
\usepackage{supertabular} 

\newtheorem{theorem}{Theorem}
\newtheorem{proposition}[theorem]{Proposition}
\newtheorem{lemma}[theorem]{Lemma}
\newtheorem{corollary}[theorem]{Corollary}

\newtheorem*{definition}{Definition}

\newtheorem*{T1}{Theorem~\ref{thm:torus}}
\newtheorem*{T2}{Theorem~\ref{thm:general}}
\newtheorem*{T3}{Theorem~\ref{thm:2q}}

\begin{document}
	
\title{Three-page indices of torus links}
	
\author[U. Jang]{Useong Jang}
\address{Department of Mathematics Education, Sunchon National University, Sunchon 57922, Korea}
\email{evanbe13134@gmail.com}

\author[M. Lee]{Minseo Lee}
\address{Department of Mathematics Education, Sunchon National University, Sunchon 57922, Korea}
\email{lms030826@naver.com}

\author[H. Yoo]{Hyungkee Yoo}
\address{Department of Mathematics Education, Sunchon National University, Sunchon 57922, Korea}
\email{hyungkee@scnu.ac.kr}

	
\begin{abstract}
An arc presentation of a link is an embedding into the open book decomposition of $\mathbb{R}^3$ with a finite number of pages.
An important rule of arc presentations is that different arcs must be placed on separate pages.
In 1999, Dynnikov proposed a three-page presentation that bends this rule by restricting the total number of pages to three.
Dynnikov showed that every link admits a three-page presentation.
In this paper, we provide an alternative proof of this result.
Also we define the three-page index $\alpha_3(L)$ of a link $L$ that the minimum number of arcs needed to represent $L$ in a three-page presentation.
We examine three-page presentations for torus links, leading to the determination of the exact three-page indices for several torus links.
\end{abstract}

\maketitle

\section{Introduction} \label{sec:intro}

A \textit{link} is a collection of mutually disjoint simple closed curves in the three dimensional space $\mathbb{R}^3$.
A \textit{sublink} of a given link is the union of some of its components.
A link is said to be \textit{splittable} if there exists a sphere that separate the link into two disjoint sublinks.
Otherwise, the link is said to be \textit{non-split}.
In particular, a single component link is called a \textit{knot}.
For any two links $L_1$ and $L_2$, if there is an ambient isotopy that transforms $L_1$ into $L_2$, then we say that two links are \textit{equivalent}.
If a link is equivalent to the union of disjoint circles on the plane, then the link is said to be \textit{trivial}.

In the cylindrical coordinate system, $\mathbb{R}^3$ can be expressed as an axis and half-planes arranged in a circular manner around the axis.
This is called an \textit{open-book decomposition} of $\mathbb{R}^3$.
In~\cite{Cr}, Cromwell showed that every link can be represented as simple arcs on each of the finite number of half planes in an open-book decomposition as drawn in Figure~\ref{fig:arc}.
This presentation is called an \textit{arc presentation}.
In this context, the axis is known as the \textit{binding axis}, each half-plane that contains an arc is called a \textit{page}, and the intersection points between a given link and the binding axis are called \textit{binding points}.
Since each arc has two endpoints and each binding point connects two arcs, the number of arcs is equal to the number of binding points.
For any link $L$, the minimal number of arcs required to construct an arc presentation of $L$ is called the arc index, denoted by $\alpha(L)$.

\begin{figure}[h!]
	\includegraphics[width=0.7\textwidth]{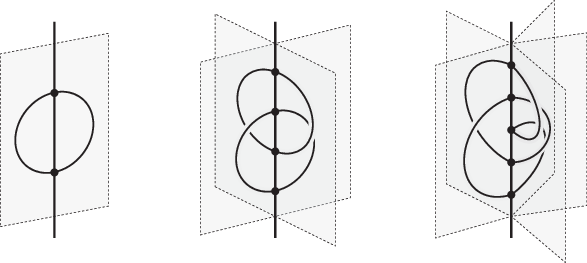}
	\caption{Arc presentations of the trivial knot, the Hopf link, and the right-handed trefoil knot}
	\label{fig:arc}
\end{figure}

An important rule of arc presentations is that different arcs must be placed on separate pages.
In this paper, we modify this rule by allowing multiple disjoint arcs to be placed on the same page, while still restricting the total number of pages.

\begin{definition}
    For some positive integer $m$, an $m$-page presentation of a link $L$ is a presentation of $L$ in the open book decomposition of $\mathbb{R}^3$, satisfying the following two conditions:
    \begin{enumerate}
        \item Exactly $m$ pages meet $L$.
        \item Each page contains properly embedded arcs.
    \end{enumerate}
    If the exact number of pages is not specified, then it is referred to as a multi-page projection.
\end{definition}

It is evident that a link $L$ has an $\alpha(L)$-page presentation, which corresponds to an arc presentation. Moreover, if a link $L$ has an $m$-page presentation, it can also have an $(m+1)$-page presentation by adding an arc on a new page. 
Our question is: what is the minimum number of pages required to represent a given link?
According to our definition, a one-page presentation is not permitted, and a two-page presentation would only represent a trivial link in the plane.
Therefore, at least three pages are required to present nontrivial links.
For example, Figure~\ref{fig:three} shows the three-page presentations of several knots and links.
In~\cite{Dyn}, Dynnikov showed that every link has a three-page presentation.
Note that the three-page presentation has been studied not only for knots and links but also for 3-valent spatial graphs and singular knots~\cite{Dyn, Dyn2, Dyn3, Kur, KV}.
Dynnikov proved this by using several local moves to transform a link diagram into a three-page presentation.
We will prove this in a different way.

\begin{figure}[h!]
	\includegraphics[width=0.7\textwidth]{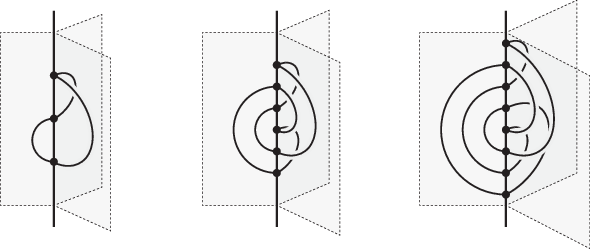}
	\caption{Three-page presentations of the trivial knot, the Hopf link, and the right-handed trefoil knot}
	\label{fig:three}
\end{figure}

Since every link has a three-page presentation, we can define a \textit{three-page index} in a similar manner to the arc index.

\begin{definition}
    For any link $L$, the three-page index $\alpha_3(L)$ of $L$ is the minimum number of arcs required to present $L$ in a three-page presentation.
\end{definition}

Equivalently, the three-page index is defined as the minimum number of binding points required for a three-page presentation.
The three-page presentation requiring the least number of arcs is called the \textit{minimal three-page presentation}.
If we distribute all the arcs from each page of a three-page presentation across different pages, then we naturally obtain an arc presentation.
Therefore, the arc index is less than or equal to the three-page index.

For any relatively prime integers $p$ and $q$,
a $(p,q)$-{\it torus knot\/} $T_{p,q}$ is a knot on an unknotted torus
that wraps around $p$ times in the meridian direction and $q$ times in the longitude direction.
In general, if $p$ and $q$ are nonzero integers with $|\gcd(p,q)| = d$, then $T_{p,q}$ consists of $d$ parallel copies of the torus knot $T_{\frac{p}{d},\frac{q}{d}}$ and is referred to as a $(p,q)$-\textit{torus link}.

The torus knot or link $T_{p,q}$ is equivalent to $T_{q,p}$. This symmetry arises because the roles of the longitudinal and meridional windings are interchangeable on the torus, resulting in identical knot structures. The mirror image of $T_{p,q}$ is given by $T_{-p,q}$, where the sign change reverses the orientation of the longitudinal winding. If either $p$ or $q$ is equal to $1$ or $-1$, then $T_{p,q}$ is trivial. Conversely, the trivial knot corresponds to these values. Note that the three-page index of a torus knot or link is invariant under symmetries of its parameters, such as the interchange of $p$ and $q$ or the reversal of their signs. Thus, without loss of generality, we may assume that $2 \leq p \leq q$.
Especially, if $p=q=n$ for some positive integer $n$, then we obtain the following theorem.

\begin{theorem} \label{thm:torus}
    For integer $n \geq 2$, three-page index of $(n,n)$-torus link $T_{n,n}$ is $4n-2$.
\end{theorem}

We examine three-page presentations for torus knots or links by using their braid forms in Section\ref{sec:torus}.

\begin{theorem} \label{thm:general}
    Let $p$ and $q$ be integers with $2 \leq p \leq q$.
    Then
    $$
    \alpha_3(T_{p,q}) \leq 2p+2q-2.
    $$
\end{theorem}

Especially, if $2p \leq q$, then we can improve the above theorem.

\begin{theorem} \label{thm:2q}
    For any integers $p$ and $q$ with $2 \leq p < 2p \leq q$,
    $$
    \alpha_3(T_{p,q}) \leq 2p+2q-3.
    $$
\end{theorem}

In~\cite{Cr}, Cromwell proposes an arc presentation for the $(p,q)$-torus link $T_{p,q}$ involving $|p|+|q|$ arcs.
That is, the arc index of the $(p,q)$-torus link is less than or equal to $|p|+|q|$.
In~\cite{Mat}, Matsuda demonstrated that equality holds when $p$ and $q$ are relatively prime.
This induces that the arc index of the nontrivial $(p,q)$-torus link is $|p|+|q|$.
Thus Theorem~\ref{thm:torus} and Theorem~\ref{thm:2q} induce the following corollaries.

\begin{corollary} \label{cor:general}
    For any nonzero integers $p$ and $q$,
    $$
    \alpha_3(T_{p,q}) \leq 2\alpha(T_{p,q})-2.
    $$
\end{corollary}

\begin{corollary} \label{cor:2q}
    For any integers $p$ and $q$ with $|q| \geq 2|p|>0$,
    $$
    \alpha_3(T_{p,q}) \leq 2\alpha(T_{p,q})-3.
    $$
\end{corollary}

This paper is organized as follows.
In Section~\ref{sec:exist}, we suggest an alternating proof for Dynnikov's result by using maximal overpasses.
In Section~\ref{sec:torus}, we transform the braid form of $T_{p,q}$ into a three-page presentation, thereby proving Theorem~\ref{thm:torus}.

\section{Alternative Proof of Dynnikov's Result} \label{sec:exist}

In a link diagram, an \textit{overpass} is a strand that passes over at least one crossing and never passes under any crossing.
If an overpass cannot be extended any longer, then it is called a \textit{maximal overpass}.
Now we prove the following theorem:

\begin{theorem}[Theorem 1 of \cite{Dyn}] \label{thm:exist}
    Every link has a three-page presentation.
\end{theorem}

\begin{proof}[Alternative proof of Theorem~\ref{thm:exist}]
Let $L$ be a link.
If every non-split link has a three-page presentation, then the three-page presentation of a splittable link can be described simply by arranging the three-page presentations of its non-split sublinks in sequence.
Thus we assume that $L$ is non-split.

Consider a link diagram for $L$.
If there are no maximal overpasses, $L$ is a trivial knot and has a three-page presentation consisting of three arcs.
If there is a maximal overpass, we perform a planar isotopy so that every overpass lies along the same line, as shown in Figure~\ref{fig:bridge}.
Note that if two or more overpasses overlap on this line, they are considered nested.
In the process of rearranging the overpasses, the non-overpass strands become properly embedded arcs in the half-planes on both sides of the line where the overpass is placed.
By pushing the overpasses into a half-plane perpendicular to the original plane, we obtain a three-page presentation of $L$.
\end{proof}

\begin{figure}[h!]
	\includegraphics[width=0.65\textwidth]{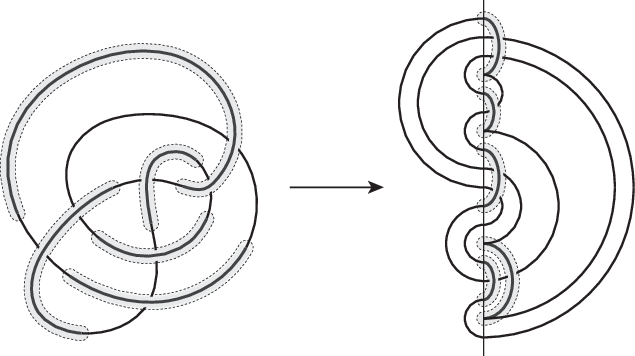}
	\caption{All maximal overpasses on the same page due to a planar isotopy}
	\label{fig:bridge}
\end{figure}

The \textit{bridge number} $br(L)$ of a link $L$ is the minimum number of overpasses in any possible link diagram for $L$.
Note that in the three-page representation, maximal overpasses are achieved by arranging two pages on a common plane while positioning the third page vertically.
Thus, we obtain the following proposition.

\begin{proposition} \label{prop:bridge}
Each page of three-page presentation of link $L$ has at least $br(L)$ arcs.
\end{proposition}

The above proposition implies the following corollary.

\begin{corollary} \label{cor:bridge}
$3 \, br(L) \leq \alpha_3(L)$.
\end{corollary}

Note that the bridge number of the Hopf link $T_{2,2}$ is two. Therefore, a three-page presentation of the Hopf link requires at least six arcs. As shown in Figure~\ref{fig:three}, the three-page presentation of the Hopf link consists of exactly six arcs. Thus, equality holds when $L$ is the Hopf link.

\section{Three-page indices for ($n$, $n$)-torus links} \label{sec:torus2}

In this section, we prove Theorem~\ref{thm:torus}. Before the proof, we observe the exceptional case as follow:

\begin{lemma} \label{lem:share}
If a pair of arcs in a three-page presentation of a link $L$ shares two common endpoints, then it is splittable.
\end{lemma}

\begin{proof}
Let $\beta$ and $\beta'$ be a pair of arcs with two common endpoints.
Denote the page containing $\beta$ by $P$ and the page containing $\beta'$ by $Q$.
Then $\beta$ and $\beta'$ together form a circle on the plane $\overline{P \cup Q}$, which bounds a disk according to the Jordan curve theorem.
We slightly push the interior of this disk into the space between pages $P$ and $Q$, avoiding intersections with other arcs.
The boundary of a sufficiently small neighborhood around this disk splits $L$ into two components.
\end{proof}

Since every torus link is a non-split link, Lemma~\ref{lem:share} tells us that each component of torus knot has at least three arcs in three-page presentation.
We already observe that the three-page index $\alpha_3(T_{2,2})$ of the Hopf link is exactly 6.
Now we recall Theorem~\ref{thm:torus}, which covers the general case.

\begin{T1}
    For integer $n \geq 2$, three-page index of $(n,n)$-torus link $T_{n,n}$ is $4n-2$.
\end{T1}  

\begin{proof}
Let $n$ be an integer greater than or equal to 2.
First, we show that there is a three-page presentation of $T_{n,n}$ with $4n-2$ arcs.
Consider the $T_{n,n}$ on the unknotted torus as drawn in Figure~\ref{fig:T_nn} (a).
If we convert the full twist into parallel kinked curves, it becomes the form shown in Figure~\ref{fig:T_nn} (b).
Take a binding axis along the thin dotted line in the figure, and we obtain the three-page presentation as shown in Figure~\ref{fig:T_nn} (c).
In this presentation, there are $2(n-1)+n+n=4n-2$ arcs.
Note that, in Figure~\ref{fig:T_nn}, the sublink consisting of the bold component and the dotted component is the Hopf link.

\begin{figure}[h!]
	\includegraphics[width=0.7\textwidth]{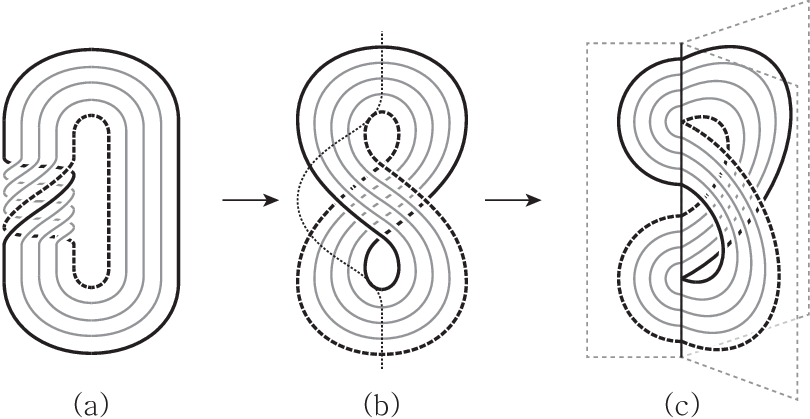}
	\caption{Three-page presentation of $T_{n,n}$ for $n\geq2$}
	\label{fig:T_nn}
\end{figure}

Now, we prove the theorem by mathematical induction.
Since $\alpha_3(T_{2,2})=6$, the base case is established.
Assume that the statement holds for some integer $m \geq 2$, so that $\alpha_3(T_{m,m})=4m-2$.
Consider the minimal three-page presentation of $T_{m+1,m+1}$.
By Lemma~\ref{lem:share}, each component of a torus link has at least three arcs in the three-page presentation.
Using the previous result, we obtain that
$$4m+1 = \alpha_3(T_{m,m})+3 \leq \alpha_3(T_{m+1,m+1}) \leq 4m+2.$$

Suppose that $\alpha_3(T_{m+1,m+1})=4m+1$.
Then there are three components $C_1$, $C_2$, and $C_3$ consisting of three arcs in the minimal three-page presentation of $T_{m+1,m+1}$.
Any pair of these three components form the Hopf link.
Without loss of generality, we take $C_1$ and $C_2$ as drawn in Figure~\ref{fig:Hopf}.

\begin{figure}[h!]
	\includegraphics[width=0.65\textwidth]{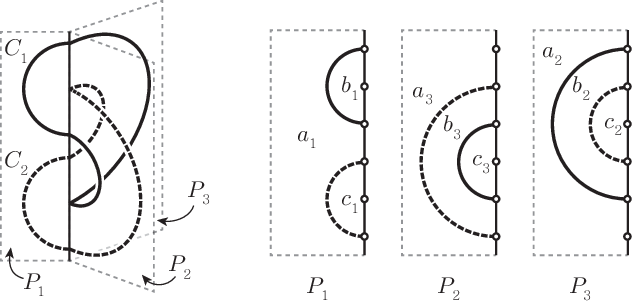}
	\caption{Three cases of $T_{3,3}$}
	\label{fig:Hopf}
\end{figure}

Let $a_i$ ($i=1,2,3$) be the unbounded region of $i$-th page $P_i$, $b_i$ be the bounded region of $P_i$ meeting $C_1$, and $c_i$ be the bounded region of $P_i$ meeting $C_2$.
The arcs of $C_3$ are placed within these regions.
Since the number of arcs in $C_3$ is three, possible combinations of regions are $\{ a_1, a_2, a_3 \}$, $\{ b_1, b_2, b_3 \}$, and $\{ c_1, c_2, c_3 \}$.
However, if $C_3$ passing through $a_1$, $a_2$, and $a_3$, then $C_3$ does not link with both $C_1$ and $C_2$.
Similarly, the remaining cases also imply that $C_3$ does not link with either $C_1$ or $C_2$.
Hence every component of $T_{m+1,m+1}$ except $C_1$ and $C_2$ has at least four arcs in the minimal three-page presentation.
This implies that $\alpha_3(T_{m+1,m+1})=4m+2$.
\end{proof}

\section{Three-page indices for general torus links} \label{sec:torus}

In this section, we prove Theorem~\ref{thm:general} and Theorem~\ref{thm:2q} using braids of links.
A \textit{braid} consists of several strands arranged so that they do not intertwine. Each strand runs from top to bottom, crossing over or under the others according to specific rules.
Typically, a braid is made up of $n$ strands, and the crossings between the strands are represented using a \textit{braid word}.
A braid word is a mathematical expression composed of symbols such as $\sigma_1, \sigma_2, \dots, \sigma_{n-1}$, where each $\sigma_i$ represents a crossing between the $i$-th and $(i+1)$-th strands. In this case, the $(i+1)$-th strand crosses over the $i$-th strand. Conversely, if the $i$-th strand crosses over the $(i+1)$-th strand, it is represented by $\sigma_i^{-1}$.
For further details, refer to ~\cite{A, Alex}.
The \textit{closure} of a braid is connecting the top and bottom ends of the strands in the braid.
The resulting knot or link after closure is called the \textit{closed braid}.
In ~\cite{Alex}, Alexander showed that every link can be expressed as the closed braid.
Now we recall Theorem~\ref{thm:general}.

\begin{T2}
    Let $p$ and $q$ be integers with $2 \leq p \leq q$.
    Then
    $$
    \alpha_3(T_{p,q}) \leq 2p+2q-2.
    $$
\end{T2}

\begin{proof}
Consider a $(p,q)$-torus link $T_{p,q}$ with $2 \leq p < q$.
The braid word of $T_{p,q}$ is the form
\begin{align*}
(\sigma_1 & \sigma_2 ... \sigma_{q-1})^{p} \\
& =(\sigma_1 \sigma_2 ... \sigma_{p-1})^{p}(\sigma_{q-p+1} \sigma_{q-p} ... \sigma_1)(\sigma_{q-p} \sigma_{q-p-1} ... \sigma_2)\cdots(\sigma_{q-1} \sigma_{q-2} ... \sigma_{p-1}).
\end{align*}
According to the braid relations, the first $q$ strings are fully twisted, while the remaining strings are untwisted.
A full twist is depicted by kinked strings.
Thus we obtain the diagram as drawn in Figure~\ref{fig:braid}.

\begin{figure}[h!]
	\includegraphics[width=0.6\textwidth]{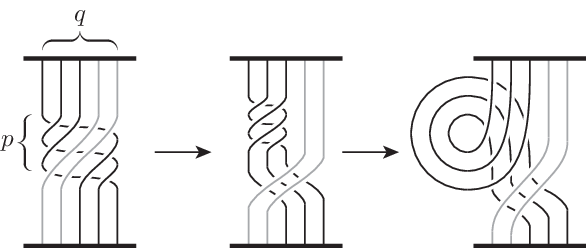}
	\caption{Braid form of $T_{p,q}$}
	\label{fig:braid}
\end{figure}

Next, we draw a line $l$ through kinked strands.
Our goal is to construct a three-page presentation with $l$ as the binding axis.
Push the strands on the bottom to the opposite side of $l$ and close the braid, and we will get a three-page presentation as drawn in Figure~\ref{fig:torus1}.
Here, there are no obstruction in the innermost arcs of the left side of $l$.
Thus we can eliminate this arcs, and hence we obtain the three-page presentation of $T_{p,q}$ with $2(p+q-1)$ arcs.

\begin{figure}[h!]
	\includegraphics[width=0.95\textwidth]{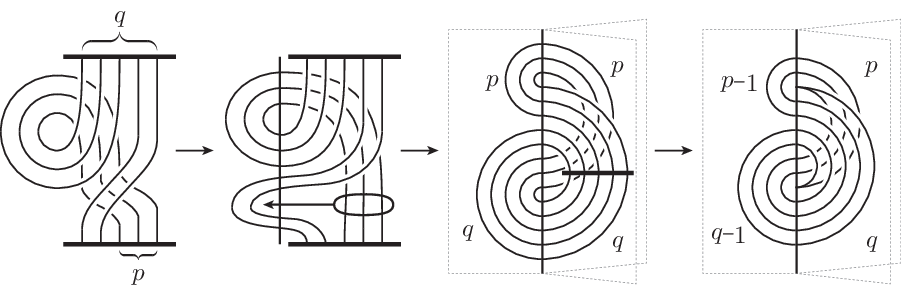}
	\caption{Three-page presentation of torus links}
	\label{fig:torus1}
\end{figure}

Now we take the another braid of $T_{p,q}$ of the form
\begin{align*}
(\sigma_1 \sigma_2 ... \sigma_{p-1})^{q} & =(\sigma_1 \sigma_2 ... \sigma_{p-1})^{\lfloor \frac{q}{p} \rfloor p}(\sigma_1 \sigma_2 ... \sigma_{q-1})^{q-\lfloor \frac{q}{p} \rfloor p}.
\end{align*}
This expression indicates that there are $\lfloor \frac{q}{p} \rfloor$ full twist of $p$ strings.
If we replace all the full twists with kinked strings and convert the remaining part as described above, we obtain a three-page presentation with $2(p+q-1)$ arcs as drawn in Figure~\ref{fig:torus2}.

\begin{figure}[h!]
	\includegraphics[width=0.95\textwidth]{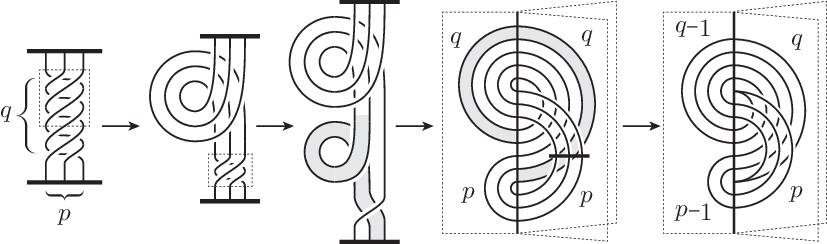}
	\caption{Three-page presentation of torus links}
	\label{fig:torus2}
\end{figure}

By the above process, we transform the braid form of $T_{p,q}$ to the three-page presentation with $2(p+q-1)$ arcs.

\end{proof}

If $q$ is greater than or equal to $2p$, then the upper bound can be improved through a simple trick.
Building on the proof of the previous theorem, we prove Theorem~\ref{thm:2q}.

\begin{T3}
    For any integers $p$ and $q$ with $2 \leq p < 2p \leq q$,
    $$
    \alpha_3(T_{p,q}) \leq 2p+2q-3.
    $$
\end{T3}

\begin{proof}
Consider a $(p,q)$-torus link $T_{p,q}$ with $2 \leq p < 2p \leq q$.
The braid word of $T_{p,q}$ is given by the following form:
$$
(\sigma_1 \sigma_2 ... \sigma_{q-1})^{p}
$$
which expands to:
$$
(\sigma_p \sigma_{p-1} ... \sigma_1)(\sigma_{p+1} \sigma_{p} ... \sigma_2) \cdots (\sigma_{q-1} \sigma_{q-2} ... \sigma_{q-p})(\sigma_{q-p} \sigma_{q-p+1} ... \sigma_{q-1})^{p}
$$
as drawn in the middle of Figure~\ref{fig:2q}.
By fixing the first strand in place and allowing the remaining $p-1$ strands to wrap around it and become kinked, the configuration can be transformed as shown on the right side of Figure~\ref{fig:2q}.

\begin{figure}[h!]
	\includegraphics[width=0.6\textwidth]{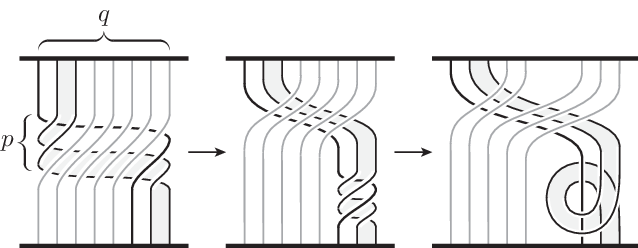}
	\caption{Three-page presentation of torus links}
	\label{fig:2q}
\end{figure}

Now, connect the $q-p$ pairs of endpoints on the left in the leftward direction, and connect the remaining $p$ pairs of endpoints from the right in the rightward direction, so that the configuration matches the left side of Figure\ref{fig:2qthree}.
Since $q$ is greater than or equal to $2p$, the right $p$ strands flow along the kinked arc.
Next, draw a straight line $l$ at the bottom of the closed braid, touching as shown in the figure.
Then, flip the kinked part to the bottom side, as shown in the middle of Figure~\ref{fig:2qthree}.
Note that arcs passing over at any crossing on the line $l$ always pass over. 
Therefore, by placing the arcs below $l$ on one page, the arcs passing over on another page, and the remaining arcs on the last page, we can obtain a three-page presentation as drawn in the right side of Figure~\ref{fig:2qthree}.

\begin{figure}[h!]
	\includegraphics[width=0.9\textwidth]{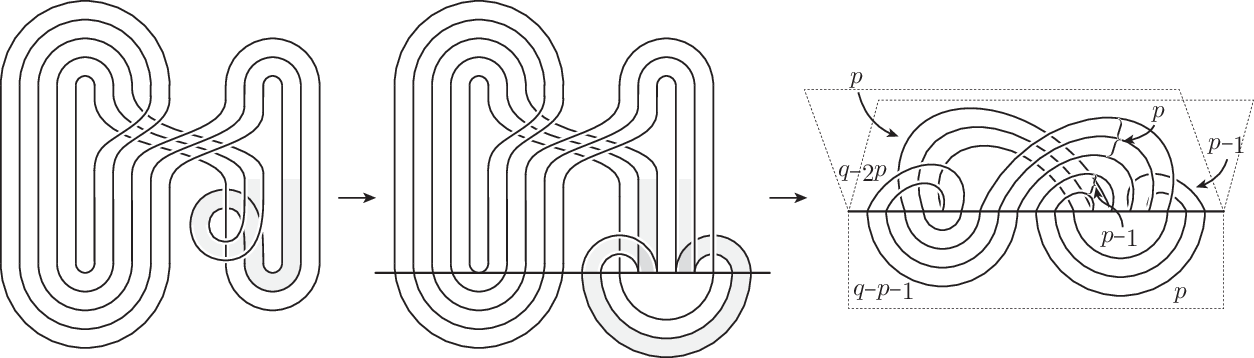}
	\caption{Three-page presentation of torus links}
	\label{fig:2qthree}
\end{figure}

On the bottom page, there are $(q-p-1) + p = q-1$ arcs; on the front page, there are $(q-2p) + (2p-1) = q-1$ arcs; and on the back page, there are $p + (p-1) = 2p-1$ arcs. Therefore, the total number of arcs is 
$$(q-1) + (q-1) + (2p-1) = 2q+2p-3.$$
Thus we obtain the result.
\end{proof}

\section*{acknowledgement}

The corresponding author(Hyungkee Yoo) was supported by Basic Science Research Program of the National Research Foundation of Korea (NRF) grant funded by the Korea government Ministry of Education (RS-2023-00244488).

\end{document}